%% file: main.tex
\documentclass[12pt,a4paper,reqno]{amsart}

\usepackage[utf8]{inputenc} 
\usepackage[IL2]{fontenc} 
\usepackage[margin=0.8in]{geometry}
\usepackage{pifont}
\usepackage[english]{babel}
\usepackage{palatino}%

\usepackage{float}
\usepackage{verbatim}
\usepackage{marvosym}
\usepackage{afterpage}

\usepackage{amsmath,amssymb,amsthm}
\usepackage{url}
\usepackage[colorlinks,linkcolor=black,urlcolor=black,citecolor=green]{hyperref}

\usepackage{pdfpages}
\usepackage[ruled,vlined]{algorithm2e}

\usepackage{tikz}
\usetikzlibrary{matrix,calc,arrows,decorations.markings,positioning,shapes.geometric}
\newcommand{\Zbb}{\mathbb{Z}}
\newcommand{\Nbb}{\mathbb{N}}
\newcommand{\Qbb}{\mathbb{Q}}
\newcommand{\Cbb}{\mathbb{C}}
\newcommand{\Rbb}{\mathbb{R}}

\newcommand{\A}{\mathcal{A}}
\newcommand{\B}{\mathcal{B}}
\newcommand{\C}{\mathcal{C}}
\newcommand{\im}{\mathrm{Im}\,}
\newcommand{\id}{\mathrm{id}}
\newcommand{\M}{\mathcal{M}}
\newcommand{\N}{\mathcal{N}}
\newcommand{\Ee}{\mathrm{E}}
\newcommand{\I}{\mathrm{I}}
\newcommand{\Ss}{\mathrm{S}}

\newcommand{\GL}{\operatorname{GL}}

\newcommand{\ef}{\operatorname{ef}}

\newcommand{\EC}{C^{\operatorname{ef}}}

\newcommand{\Hom}{\operatorname{Hom}}

\renewcommand{\d}[1]{\ensuremath{\operatorname{d}\!{#1}}}

\newcommand*\circled[1]{\tikz[baseline=(char.base)]{
            \node[shape=circle,draw,inner sep=2pt] (char) {#1};}}

\theoremstyle{plain}
\newtheorem{thm}{Theorem}[section]

\newtheorem{lem}[thm]{Lemma}
\newtheorem{prop}[thm]{Proposition}
\newtheorem{cor}[thm]{Corollary}

\newtheorem{constr}[thm]{Construction}
\newtheorem{innercustomthm}{Theorem}
\newenvironment{mythm}[1]
  {\renewcommand\theinnercustomthm{#1}\innercustomthm}
  {\endinnercustomthm}

\theoremstyle{definition}
\newtheorem{defn}[thm]{Definition}
\newtheorem*{exmp}{Example}
\theoremstyle{remark}
\newtheorem*{rem}{Remark}

\begin{document}
\title[Rational homotopy equivalence]{Rational Homotopy Equivalence\\ -- An Algorithmic Approach}
\begin{abstract}
This article proposes an algorithm that constructs a Sullivan minimal model for any simply connected simplicial set with effective homology and thereby allows one to  decide algorithmically whether two simply connected spaces represented by finite simplicial sets have the same rational homotopy type.
\end{abstract}
\author{M\'{a}ria \v{S}imkov\'{a}}
\address{Department of Mathematics and Statistics, Masaryk University, Kotl\'{a}\v{r}sk\'{a} 2, 611 37 Brno,  Czech Republic}
\curraddr{}
\email{simkova@math.muni.cz}
\thanks{}

\date{\today}

\dedicatory{}

\maketitle

\input{01Intro}
\input{02SSetsAndEffHom}
\input{03DGAs}
\input{04AlgGroups}

\input{05aRatHomTh}

\input{05bRHTandDGA}
\input{05cMinModAndPostnikTow}
\input{06MinModForSSet}
\input{07mainThm}

\bibliographystyle{amsalpha}
\bibliography{bibtex}

\end{document}

%% file: 01Intro.tex
\section{Introduction}
The main goal of this article is to present an algorithm that finds a minimal Sullivan model for 
a simply connected simplicial set with effective homology. (Any finite simplicial complex is a special case.) Using this construction and an algorithm on orbits and stabilizers of algebraic groups, one can algorithmically decide whether two finite simply connected  simplicial sets have the same rational homotopy type.

In \cite{MB} the authors describe an algorithm that finds a minimal Sullivan model of any simply connected differential graded algebra which is finitely presentable with generators of positive degree. The construction follows the classical proof of the existence of such a minimal model; see Proposition 12.2 in \cite{FHT}. However, this algorithm cannot be used to compute a minimal Sullivan model of a simply connected finite simplicial set $X$, since the differential graded algebra $\mathcal A_{PL}(X)$ has generators of degree $0$ and is not usually given as a quotient  of a free algebra by a differentiable ideal. Therefore, we first generalize their algorithm to a class of more suitable differential graded algebras. These are differential graded algebras that are strongly homotopy equivalent to chain complexes which are finitely generated in all degrees
(see Chapter 2).

In the next step, we show that the differential graded algebra $\mathcal A_{PL}(X)$ of a simplicial set $X$ with effective homology satisfies the assumptions of our algorithm. To do this, we use Dupont's result \cite{JD2} on the reduction of the algebra $\mathcal A_{PL}(X)$ to a cochain complex of $X$ with rational coefficients and show that all mappings that appear in the reduction are algorithmically computable.

Finally, we exploit the fact that two simply connected spaces with finitely generated homotopy groups have the same rational homotopy type if and only if their minimal Sullivan models are isomorphic. We translate the decision whether two simply connected simplicial sets have the same rational homotopy type into the decision whether their minimal Sullivan models are isomorphic. We solve this problem using Sarkisyan's algorithm on algebraic $\mathbb Q$-groups, see \cite{RAS}.

The inspiration for our research was the article by Nabutovsky and Weinberger
\cite{NW}, in which they argue that the question of whether two simply connected spaces are homotopy equivalent is algorithmically solvable. They propose to use rational homotopy theory for the algorithm, but do not provide any description of the algorithm. This article, together with the article \cite{MS2}, paves the way to the realization of their idea.

The paper starts with a quick recapitulation of basic concepts on simplicial sets and the effective homology framework, which is behind the algorithmic approach to homotopy theory. In Section 3 we recall notions concerning differential graded algebras, mi\-ni\-mal algebras, and minimal Sullivan models and describe the algorithm which assigns a minimal model to every differential graded algebra that is strongly homotopy equivalent to a cochain complex with finite homology (Theorem A). The next Section deals with algebraic groups and automorphisms of differential graded algebras. Section 5 is devoted to rational homotopy theory. We recall the notion of homotopy in DGAs and summarize existing results related to a correspondence of rational homotopy equivalences and isomorphisms of minimal models. At the end of this section, we define Hirsch extensions of a DGA and their relation to minimal models of Postnikov stages. The aim of the following section is to recall a version of the de Rham Theorem stating that the algebra of polynomial differential forms on a simplicial set can be reduced to the cochain complex of this set with rational coefficients. In the last section, we describe an algorithm assigning a minimal model to a finite simply connected simplicial set (Theorem C) and design our main algorithm which decides if two finite simply connected simplicial sets are rationally homotopy equivalent (Theorem D).

%% file: 02SSetsAndEffHom.tex
\section{Preliminaries on simplicial sets with effective homology}

We will work with simplicial sets. For basic information on simplicial sets, we refer to the comprehensive sources \cite{M, C, GJ}. In this section, we summarize the results we need for an algorithmic approach to the homotopy theory of simplicial sets.

\subsection*{Simplicial sets}
The standard (geometric) $n$-simplex is the set
\[\Delta^n=\{(x_0,x_1,\dots,x_n)\in \mathbb R^{n+1};\ \sum_{i=0}^n x_i=1,\ x_i\ge 0 \text{ for all }i \}. \] 
Its vertices will be denoted $e_0, e_1,\dots,e_n$. The symbol $\Delta[n]$ stands for the simplicial set whose $k$-simplices are $(k+1)$-tuples $(i_0,i_1,\dots, i_k)$ of integers such that
$0\le i_0\le i_1\le \dots \le i_k\le n$. The geometric realization of $\Delta[n]$ is $\Delta^n$.

Let $\pi$ be an abelian group.
In the context of simplicial sets, the Eilenberg-MacLane simplicial set $K(\pi,n)$ is defined through its standard minimal model, in which $k$-simplices are given by cocycles on nondegenerate $n$-simplices of $\Delta[k]$:
\[K(\pi,n)_k=Z^n(\Delta[k],\pi).\]
It will be essential that this model of $K(\pi,n)$ is a simplicial group and hence a Kan complex. Similarly, we define a simplicial set $E(\pi,n)$ whose $k$-simplices are given by cochains: 
\[E(\pi,n)_k=C^n(\Delta[k],\pi).\]
The previous definitions lead to a natural principal fibration known as Eilenberg-MacLane fibration $\delta\colon E(\pi,n)\to K(\pi,n+1)$ for $n\geq 1$. These fibrations play an important role in the construction of Postnikov towers.

\begin{defn}
Let $Y$ be a simplicial set. A simplicial Postnikov tower  for $Y$ is the following collection of mappings and simplicial sets organized into the commutative diagram
\begin{equation}
\begin{tikzpicture}
\matrix (m) [matrix of math nodes, row sep=2em,
column sep=2em, minimum width=2em]
{  &[2cm] Y_n\\
    &[2cm] \vdots  \\
   &[2cm] Y_1 \\
  Y &[2cm] Y_0\\
};
  \begin{scope}[every node/.style={scale=.8}]
\path[->](m-4-1) edge node[above] {$\varphi_0$} (m-4-2);
\path[->](m-4-1) edge node[above] {$\varphi_1$} (m-3-2);
\path[->](m-4-1) edge node[auto] {$\varphi_n$} (m-1-2);
\path[->](m-3-2) edge node[auto] {$p_1$} (m-4-2);
\path[->](m-2-2) edge node[auto] {$p_2$} (m-3-2);
\path[->](m-1-2) edge node[auto] {$p_n$} (m-2-2);
\end{scope}
\end{tikzpicture}\label{3d2}
\end{equation}
such that for each $n\geq0$ the map $\varphi_n\colon Y\to Y_n$
induces isomorphisms $\varphi_{n*}\colon \pi_k(Y)\to \pi_k(Y_n)$
of homotopy groups with $0\leq k\leq n$, and $\pi_k(Y_n)=0$ for $k\geq n+1$.
The simplicial set $Y_n$ is called the $n$-th Postnikov stage. 
\end{defn}

\begin{defn}\label{3d1}
Let $Y$ be a simply connected simplicial set. A standard Postnikov tower is a Postnikov tower such that  $Y_n$ is the pullback of the fibration $\delta$ along a map $k_{n-1}\colon Y_{n-1}\to K(\pi_n(Y),n+1)$ for all $n\geq 1$:
\begin{center}
\begin{tikzpicture}
\matrix (m) [matrix of math nodes, row sep=2em,
column sep=2em, minimum width=2em]
{  Y_n & E(\pi_n(Y),n)  \\
  Y_{n-1} & K(\pi_n(Y),n+1) \\
};
  \begin{scope}[every node/.style={scale=.8}]
\path[->](m-1-1) edge node[above]{$r_n$}  (m-1-2);
\path[->](m-1-1) edge node[left] {$p_n$} (m-2-1);
\path[->](m-2-1) edge node[below] {$k_{n-1}$} (m-2-2);
\path[->](m-1-2) edge node[right] {$\delta$} (m-2-2);
\end{scope}
\end{tikzpicture}
\end{center}
The map $k_{n-1}$ is called a Postnikov map. Since $p_n$ are pullbacks of the Kan fibrations $\delta$, they are also Kan fibrations.
\end{defn}

\subsection*{Effective homology}
Here we look at the basic notions of the effective homology framework. 
This paradigm was developed by Sergeraert and his collaborators to deal with infinitary objects, see \cite{RS} (or \cite{poly}) for more details. \\

A locally effective simplicial set is a simplicial set whose simplices
have a specified finite encoding, and whose face and degeneracy operators are specified by algorithms.

We will work with non-negatively graded chain or cochain complexes of free abelian groups or $\Qbb$-vector spaces. Such a chain complex is locally effective if elements of the graded module can be represented in a computer and the operations of zero, addition, and differential are computable.\\

In all parts of the paper where we deal with algorithms, all simplicial sets are locally effective, and all chain complexes
are non-negatively graded locally effective chain complexes of free $\Zbb$-modules or $\Qbb$-vector spaces. All simplicial maps, chain maps, chain homotopies, etc., are computable.\\

An effective chain complex is a (locally effective) free chain complex equipped
with an algorithm that generates a list of elements of the distinguished basis in any given dimension (in particular, the distinguished bases are finite in each dimension).

\begin{defn}
[\cite{RS}]
Let $(C,d_C)$ and $(D,d_D)$ be chain complexes. A triple of mappings $(f\colon C\to D,g\colon D\to C,h\colon C\to C)$ is called a reduction if the following holds
\begin{itemize}
\item[i)] $f$ and $g$ are chain maps of degree $0$,
\item[ii)] $h$ is a map of degree 1,
\item[iii)] $fg=\id_D$ and $\id_C-gf=[d_C,h]=d_Ch+hd_C$,
\item[iv)] $fh=0$, $hg=0$ and $hh=0$ known as side conditions.
\end{itemize}
The reductions are denoted as $(f,g,h)\colon (C,d_C)\Rightarrow (D,d_D)$.  If we deal with cochain complexes, $d_C$ and $d_D$ have degree $1$ and $h$ is a map of degree $-1$.)

\begin{rem}
If a triple of maps $(f,g,h)$ satisfies only conditions i) to iii) we can change the maps such that new maps satisfy not only conditions i), ii), iii) but also the side conditions iv). See 2.1 of \cite{LS}. We need these conditions to harness the power of effective homology framework. 
\end{rem}

A strong homotopy equivalence $C\Longleftrightarrow D$ between chain complexes $C$, $D$ is a chain complex $E$ together with a pair of reductions $C\Leftarrow E\Rightarrow D$.

Let $C$ be a chain complex. We say that $C$ is equipped with effective
homology if there is a specified strong equivalence $C\Longleftrightarrow \EC$
of $C$ with some effective chain complex $\EC$.

Similarly, we say that a simplicial set has (or can be equipped with) effective homology if its chain complex generated by nondegenerate simplices is equipped with  effective homology.
\end{defn}


\begin{lem}\label{HomEf}
Any strong homotopy equivalence $C\Longleftrightarrow D$ between chain complexes $C_*$ and $D_*$ of free Abelian groups induces strong homotopy equivalences $C\otimes\Qbb \Longleftrightarrow D\otimes\Qbb$ and  $\Hom(C,\Qbb)\Longleftrightarrow \Hom(D,\Qbb)$ of chain and cochain complexes of $\Qbb$-vector spaces, respectively.
\end{lem}

\begin{lem}\label{CHE}
Let $C\Longleftrightarrow D$ be a strong homotopy equivalence between  chain complexes. Then there are maps
$f\colon C_*\to D_*$, $g\colon D_*\to C_*$, $h_C\colon C_*\to C_{*+1}$ and $h_D\colon D_*\to D_{*+1}$ such that
\[ gf-\id_C=d_Ch_C+h_Cd_C\quad\text{and}\quad fg-\id_D=d_Dh_D+h_Dd_D.\]
\end{lem}

It is clear that all finite simplicial sets have effective homology. It is essential from the algorithmic point of view that many infinite simplicial sets also have effective homology. 

\begin{prop}[\cite{poly}, Section 3]
Let $n\geq 1$ be a fixed integer and $\pi$ a finitely generated abelian group. The standard simplicial model of the Eilenberg-MacLane space can be equipped with effective homology.

If $P$ is a simplicial set  equipped  with  effective  homology and $f\colon P\to K(\pi,n+1)$ is  computable, then  the  pullback $Q$ of $\delta\colon E(\pi,n)\to K(\pi,n+1)$ along $f$ can  be  equipped  with  effective homology.
\end{prop}

In \cite{poly}, Section 4, one can find the algorithmic construction of a standard Postnikov tower for any simply connected simplicial set with effective homology. A shorter summary of the construction can also be found in \cite{MS2}, Section 3. Due to the previous proposition, the constructed Postnikov stages have effective homology. The Postnikov tower obtained by the construction is called effective.

%% file: 03DGAs.tex
\section{DGAs and minimal models}
In this section we recall the basic notions concerning differential graded algebras. We emphasize that in the paper all vector spaces and algebras are over the field $\Qbb$. 
\begin{defn}
A commutative cochain algebra (or differential graded algebra or DGA for short) is a graded vector space
 \[ \A = \bigoplus_{p\geq0}\A^p\]
together with 
\begin{itemize}
    \item a multiplication $\cdot\colon \A^p\otimes\A^q\to\A^{p+q}$ satisfying $a\cdot b=(-1)^{pq}b\cdot a$,
    \item a differential $d\colon \A^p\to\A^{p+1}$ satisfying $d^2=0$
    and the Leibnitz rule 
     \[d(a\cdot b)=d(a)\cdot b+(-1)^pa\cdot d(b).\]
\end{itemize}
If $\A^0=\Qbb$, then we say that $\A$ is connected.
If, moreover, $H^0(\A)=\Qbb$ and $H^1(\A)=0$, $\A$ is called simply connected.
\end{defn}

\begin{defn}%
A DGA $\A$ is said to be minimal if
\begin{itemize}
    \item $\A$ is free as a graded-commutative algebra on generators of degrees $\geq 2$,
    \item $d$ is decomposable i.e. $d(\A^+)\subset \A^+\cdot\A^+$ where $+$ denotes positive degree.
\end{itemize}
A DGA is free as a graded-commutative algebra if it is a tensor product of polynomial algebras on generators of even degrees and exterior algebra on generators of odd degrees. 
\end{defn}

A homomorphism of DGAs that induces isomorphism in cohomology is called a quasi-isomorphism. 

The essential notion for us is the notion of a minimal model for a DGA.
\begin{defn}
A  minimal model for a DGA $\A$ is a quasi-isomorphism $m\colon\M\to \A$ such that $\M$ is a minimal DGA. 
\end{defn}

\begin{prop}[Proposition 12.2 in \cite{FHT}]\label{emm}
For every DGA ($\A$, $d$) such that $H^0(\A)=\Qbb$  and  $H^1(\A)=0$, there is a minimal model $m\colon (\M,d)\to(\A,d)$.
\end{prop}

The purpose of this section is to describe an algorithm which constructs a minimal model for
a certain class of differential graded algebras. But before that, let us recall the concept of homotopy between two homomorphisms of DGAs.

\begin{defn}
Let $\A$ and $\B$ be DGAs. Homomorphisms $f,g\colon\A\to\B$ are homotopic if there exists a homomorphism of DGAs
$H\colon \A\to\Lambda(t,dt)\otimes\B$ such that the projections for $t=0$ and $t=1$ are equal to $f$ and $g$, respectively. We denote this relation with $\sim$. 
\end{defn}

The following lemma provides a sufficient condition for $\sim$ to be an equivalence relation.
\begin{lem}[Corollary 11.4 in \cite{GM}]
Let $\A$ be a DGA and $\M$ be a minimal DGA. The relation of being homotopic for homomorphisms from $\M$ to $\A$ is an equivalence relation.    
\end{lem}
Let $\M$ be a minimal model and $\A$ a DGA. The set of homotopy classes of homomorphisms $\M\to\A$ is denoted by $[\M,\A]$.

The next three statements summarize key results from \cite{GM}.
\begin{lem}[Lemma 11.7 in \cite{GM}]\label{QIsoIso}
Let $\M$ and $\M'$ be minimal DGAs, and suppose that $\varphi\colon\M\to\M'$ is a quasi-isomorphism. Then, $\varphi$ is an isomorphism.
\end{lem}

\begin{prop}[Lemma 11.5 in \cite{GM}]\label{QIsoBij}
Let $\varphi\colon\B \to\C$ be a quasi-isomorphism of DGAs and let $\M$ be a minimal DGA. Then $\varphi_*\colon [\M,\B]\to[\M,\C]$ is a bijection.
\end{prop}

\begin{cor}[Theorem 11.6 in \cite{GM}]\label{unimm}
Let $\A$ be a DGA and $m\colon\M\to\A$ and $m'\colon\M'\to\A$ be two minimal models for  $\A$. Then, there is up to homotopy just one isomorphism $\varphi\colon\M\to\M'$
such that $m'\circ\varphi\sim m$.
\end{cor}

 The paper by V. Manero and M. M. Buzunáriz \cite{MB} presents
a method to compute a minimal model for a finitely presented DGA up to a specified degree, together with a map that is a quasi-isomorphism up to the given degree. The method works by adding generators one by one. It terminates
if and only if the minimal model is finitely generated up to the given degree. Our intention is to extend the result to a broader set of input objects. The algorithm adjusts the classical result showing minimal model existence; see Proposition \ref{emm}. 

\begin{defn}
 A DGA $\M$ is said to be $n$-minimal if it is minimal and all generators are of degree at most $n$. The DGAs $\A$ and $\B$ are $n$-quasi-isomorphic if there exists a morphism of
DGAs $f\colon\A\to\B$ such that $f^* \colon H^j (\A) \to H^j (\B)$ is an isomorphism for every $j \leq n$ and
 $f^* \colon H^{n+1} (\A) \to H^{n+1} (\B)$ is a monomorphism. An $n$-minimal model of the DGA $\A$ is an $n$-minimal differential algebra
$\M$ together with an $n$-quasi-isomorphism $f\colon\M\to\A$.
\end{defn}

The following theorem extends the paper \cite{MB}; we use similar methods to prove the statement.

\begin{mythm}{A}\label{minMod}
There is an algorithm which, given $n\in\Nbb$ and a simply-connected DGA $\A$ equipped with a strong homotopy equivalence to a cochain complex $C^*\otimes\Qbb$ where $C^*$ is an effective cochain complex of abelian groups with finite dimensional homology, constructs an $n$-minimal model for $\A$.
\end{mythm}

\begin{proof}
We will search for the minimal model in the form of the free algebra $\Lambda V$ on a graded vector space $V$. The subspace of elements of degree $\le k$ will be denoted by $V^{\le k}$.

Note that we can always compute a finite list of representatives of cohomology generators of $\A$ using the cohomology of the effective chain complex $C^*\otimes\Qbb$. Let us denote the components of the provided strong equivalence $\A^*\Longleftrightarrow C^*\otimes\Qbb$ by $f\colon\A^*\to C^*\otimes\Qbb$, $g\colon C^*\otimes\Qbb\to \A^*$ and $h\colon\A^*\to \A^{*-1}$, see Lemma \ref{CHE}.\\

We proceed by induction. Since $\A$ is simply connected, the induction starts at degree 2.
Let us define $V^2=V^{\le 2}$ as a rational vector subspace of $\A$ generated by a finite list of representatives $v_{1},\dots, v_{l}$  of the basis of $H^2(\A)$. Let $m_2\colon (\Lambda V^2,0)\to(\A,d)$ be the extension of the inclusion $V^2\to\A$. Then $H^1(m_2)=0$ is an isomorphism,
$H^2(m_2)$ is an isomorphism, since $V^2\cong H^2(\A)$, and $H^3(m_2)$ is injective since $(\Lambda (V^2)^3$ does not have elements of degree 3. Clearly, $\dim V^2<\infty$.\\

In the induction step, suppose that $m_k\colon (\Lambda V^{\leq k},d)\to (\A,d)$ is constructed such that $H^i(m_k)$  is an isomorphism for $i\leq k$, $H^{k+1}(m_k)$ is a monomorphism and $\dim V^{\leq k}<\infty$. We extend it to $m_{k+1}\colon  (\Lambda V^{\leq {k+1}},d)\to (A,d)$ with the same properties. 

First, we find a finite list of elements $w_p\in \A^{k+1}$ such that
\[H^{k+1}(\A)=\im H^{k+1}(m_k)\oplus\bigoplus_p\Qbb[w_p].\]
We can do it in the following steps:
\begin{itemize}
 \item Compute representatives $u_j^{k+1}\in \A^{k+1}$ of the basis of $H^{k+1}(\A)$ using a strong equivalence between $\A^*$ and $C^*\otimes\Qbb$ with effective homology.  
  \item Compute a finite basis  $g_i^{k+1}$ of $\Qbb$-vector subspaces $\ker d^{k+1}\subseteq(\Lambda V^{\leq k})^{k+1}$.
  \item Since $H^{k+1}(m_k)$ is a monomorphism, the elements $m_k(g_i^{k+1})$ are representatives
  of the basis of $\im H^{k+1}(m_k)\subseteq H^{k+1}(\A)$. We want to express $m_k(g_i^{k+1})$ as linear combinations of $u_j^{k+1}$ modulo $\im d^k$ in $\A^{k+1}$.  This means to find $a_i\in\A^k$ and $\gamma^i_j\in\Qbb$ such that
    \[m_k(g_i^{k+1}) = \sum_j \gamma^i_j u^{k+1}_j + d^k(a_i).\]
    \item As $\A^{k+1}$ could be infinitely generated, we have to transfer computation to the effective cochain complex $C^{*}\otimes\Qbb$
      \[f(m_k(g_i^{k+1})) = \sum_j \gamma^i_j f(u^{k+1}_j) + d^k(f(a_i)).\]
      Write $f(a_i) = \sum_l \epsilon^i_lc_l^{k}$, where $c_l^k\in C^{k}$ are basis elements, and solve the equation for unknowns $\gamma^i_j,\epsilon^i_l\in\Qbb$
            \[f(m_k(g_i^{k+1})) = \sum_j \gamma^i_j f(u^{k+1}_j) + \sum_l \epsilon^i_ld^k(c_l^{k}).\]
            in the finite dimensional space $C^{k+1}\otimes \mathbb Q$.
     \item Go back to $\A^{k+1}$ by applying the map $g$ to the preceding equation 
    \[ gf(m_k(g_i^{k+1})) = \sum_j \gamma^i_j gf(u^{k+1}_j) + \sum_l \epsilon^i_ld^kg(c_l^{k}).\]
    Since $gf=\operatorname{id}+hd^{k+1}+d^kh$, we get
     \begin{align*}
     gf(m_k(g_i^{k+1}))&=m_k(g_i^{k+1})+hd^{k+1}(m_k(g_i^{k+1}))+d^{k}h(m_k(g_i^{k+1}))\\
       &=m_k(g_i^{k+1})+d^{k}h(m_k(g_i^{k+1})).
     \end{align*}
     and 
     \[ gf(u^{k+1}_j) =u^{k+1}_j + hd^{k+1}(u^{k+1}_j) + d^kh(u^{k+1}_j)= u^{k+1}_j  + d^kh(u^{k+1}_j) \]
     that implies
    \[
     m_k(g_i^{k+1})+d^{k}h(m_k(g_i^{k+1}))=\sum_j\gamma^i_j \left(u_j^{k+1}+d^{k}h(u_j^{k+1})\right)+\sum_l \epsilon^i_ld^kg(c_l^{k}).
     \]
     and so    
        \[
     m_k(g_i^{k+1}) = \sum_j \gamma^i_j u^{k+1}_j + d^k\left(\sum_j \gamma^i_j h(u^{k+1}_j )+\sum_l \epsilon^i_lg(c_l^{k})- h(m_k(g_i^{k+1}))\right).
     \]
     \item Thus, we get $[m_k(g^{k+1}_i)]\in H^{k+1}(\A)$ as linear combinations $\sum_j \gamma^i_j [u^{k+1}_j]$. Hence, using cohomology classes $[u^{k+1}_j]$ we can complete the elements $[m_k(g^{k+1}_i)]$ to the basis of $H^{k+1}(\A)$. These $u^{k+1}_j$ are required $w_p$.
     \end{itemize}

Next, we find a finite list of elements $z_p\in \left(\Lambda(V^{\le k})\right)^{k+2}$ such that    
\[\ker H^{k+2}(m_k)=\bigoplus_q \Qbb[z_q]\]
as follows:
\begin{itemize}
        \item Compute a finite list of basis elements $g_i^{k+2}$ of $\ker d^{k+2}\subseteq(\Lambda V^{\leq k})^{k+2}$.
    \item A linear combination $\sum_i\beta_i g_i^{k+2}$ represents an element of $\ker H^{k+2}(m_k)$
     if and only if there is an element $a^{k+1}\in\A^{k+1}$ such that
    \[\sum_i \beta_i m_k(g_i^{k+2})= d^{k+1}(a^{k+1})\in\A^{k+2}\]
    \item We again transfer this equation in unknowns $\beta_i\in\mathbb Q$ and $a^{k+1}\in \A^{k+1}$ to 
    $C^*\otimes\mathbb Q$
    \[\sum_i \beta_i fm_k(g_i^{k+2})= d^{k+1}(f(a^{k+1}))\in C^{k+2}\otimes \Qbb.\]
    Write $f(a^{k+1})=\sum_j \alpha_jc_j^{k+1}$ with basis elements $c^{k+1}_j$ of $C^{k+1}$ and unknowns $\alpha_j\in\Qbb$. Now, the equation 
     \[\sum_i \beta_i fm_k(g_i^{k+2})= \sum_j \alpha_j d^{k+1}(c_j^{k+1})\in C^{k+2}\otimes \mathbb Q\]
     in unknowns $\beta_i$ and $\alpha_j$ is easily solvable. 
     \item Move back to $\A^{k+2}$ by applying the mapping $g$ to above equation
         \[\sum_i \beta_i gf(m_k(g_i^{k+2}))= \sum_j \alpha_j d^{k+1}g(c_j^{k+1})\in \A^{k+2}.\]
     Using again $gf=\operatorname{id}+hd^{k+2}+d^{k+1}h$ we get 
        \[
        \sum_i \beta_i(m_k(g_i^{k+2})+d^{k+1}h(m_k(g_i^{k+2})))= \sum_i \beta_i gf(m_k(g_i^{k+2}))= \sum_j \alpha_j d^{k+1} g(c_j^{k+1})\in\A^{k+2}\]
        which leads to
         \[\sum_i \beta_i m_k(g_i^{k+2}) = d^{k+1}\left(\sum_j \alpha_j g(c_j^{k+1})-\sum_i\beta_ih(m_k(g_i^{k+2}))\right).\] 
         The representatives $z_q\in (\Lambda V^{\le k})^{k+2}$ of the basis of $\ker H^{k+2}(m_k)$ are linear combinations $\sum_i \beta_ig_i^{k+2}$ with the concrete $\beta_i$ calculated in the effective cochain complex $C^{k+2}\otimes\Qbb$. For each of them there is a corresponding element $b_q = \sum_j \alpha_j g(c_j^{k+1})-\sum_i\beta_ih(m_k(g_i^{k+2}))\in \A^{k+1}$ such that $m_k(z_q) = d^{k+1} b_q$. There are only finitely many $z_q$ as we have a linear system over the finite-dimensional vector space
        $C^{k+2}\otimes\mathbb Q$.
 \end{itemize}

At this stage, we have found elements $w_p$ and $b_q$ with which we can continue the inductive construction of a minimal model. Define $V^{k+1}$ as the vector space of degree $k+1$ generated by the basis $\{w'_p,b'_q\}$. Put $V^{\leq k+1}=V^{\leq k}\oplus V^{ {k+1}}$ and extend $m_k$ to the morphism $m_{k+1}:V^{\le k+1}\to \A$  
\[m_{k+1}w'_p=w_p,\quad  m_{k+1}b'_q=b_q.\]
Extend $d$ to a derivation in $\Lambda V^{\leq k+1}$ 
\[dw'_p=0,\quad db'_q=z_q .\]
By construction, $d^2=0$ in both $V^{k+1}$ and $\Lambda V^{\leq k}$, thus $d^2=0$ in $\Lambda V^{k+1}$. Similarly, $m_{k+1}d=dm_{k+1}$ in $V^{k+1}$  and in $\Lambda V^{\leq k}$, and so $m_{k+1}d=dm_{k+1}$ in $\Lambda V^{\leq k+1}$. We added only a finite number of generators, so we have $\dim V^{k+1}<\infty$ again.

By induction $H^i(m_{k+1})=H^i(m_k)$ is an isomorphism for $i\le k$. We have chosen $[m_k(g_i^{k+1})]$ and [$w_p$] to form a basis of $H^{k+1}(\A)$. Since $d(w'_p)= 0$,   a basis of $H^{k+1}(\Lambda V^{\le k+1})$ is determined by elements $[g_i^{k+1}]$ and [$w'_p$].  Hence $H^{k+1}(m_{k+1})$ is an isomorphism. 

Finally, $\ker H^{k+2}(m_{k+1})$ is generated by elements $[z_q]$. Since 
\[m_{k+1}(z_q)=m_{k+1}(db'_q)=dm_{k+1}(b'_q),\]
the kernel is zero and $H^{k+2}(m_{k+1})$ is a monomorphism.
\\

We terminate our construction at the required degree of $n$.
\end{proof}

%% file: 04AlgGroups.tex
\section{Algebraic groups and automorphisms of DGAs}
Can we algorithmically decide whether two minimal models are isomorphic? In solving this problem, the fact that the automorphisms of the minimal model form an explicitly given algebraic group plays an important role.

\begin{defn}
A subgroup $G$ of $\GL(n,\Cbb)$ is an algebraic matrix group of degree $n$ defined over $\Qbb$ if it is the set of common zeros in $\GL(n,\Cbb)$ of finitely many polynomials $p_1,\dots, p_k\in\Qbb[x_{11},x_{12},\dots,x_{nn}]$ where $x_{11}, x_{12},\dots, x_{nn}$ represent matrix entries.  In short, we will call the group a $\Qbb$-group of degree $n$. Next we define
\[ G_{\mathbb Q}=G\cap\GL(n,\mathbb Q).\]
The $\Qbb$-group $G$ is given explicitly if the polynomials $p_1,\dots, p_k$ are explicitly given.
\end{defn}

\begin{defn}
Let $G$ be a $\Qbb$-group and $W\subseteq\Cbb^m$ be a vector subspace.
A homomorphism $\rho\colon G\to GL(W)$ is a {\it rational linear representation}
if the components of $\rho(g)\in GL(W)$ 
are rational functions in the entries of the matrix $g$. 

The homomorphism $\rho$ determines the right action of the group $G$ on $W$, $w\longmapsto w\cdot \rho(g)$, and moreover, for $w\in W\cap \mathbb Q^m$ and $g\in G_{\mathbb Q}$ the vector $w\cdot\rho(g)\in \mathbb Q^m$.

We say that the rational linear representation $\rho$ is {\it explicitly given} if
there is an effective procedure which for each $w\in W$  and each $g\in G$ produces all components of $w\cdot \rho(g)\in W$. 
\end{defn}

For our purposes, we will need the following result by R.~A.~Sarkisyan in the special case
of the algebraic number field $\Qbb$.

\begin{thm}[Theorem 4.1 in \cite{RAS}]\label{algC}
There exists an algorithm which for any explicitly given $\Qbb$-group $G$, any explicitly defined right representation $\rho\colon G\to\GL(m,\Cbb)$ and any two fixed vectors $x,y\in \Qbb^m$ verifies the existence of an element $g\in G_\Qbb$ with the property
\[ x\cdot\rho(g)=y\]
and explicitly finds some such element $g\in G_\Qbb$.
\end{thm}

The author proved this theorem in \cite{RAS} under the assumption that for
$H^1(\mathbb Q,G)$  the Hasse principle holds. The results of F. Grunewald and O. Segal \cite{GS} suggested a method that allowed him to abandon this assumption. More details are available in \cite{RAS3}. 

If we accept the definition of an algebra as a vector space with bilinear multiplication, the previous result gives the following:

\begin{cor}[Theorem 8.1 in \cite{RAS2}] The isomorphism problem for finite-di\-men\-sio\-nal algebras over $\Qbb$ is algorithmically decidable. 
\end{cor}

In a similar way, we use Sarkisyan's theorem to prove:

\begin{mythm} B\label{IsoMod}
There is an algorithm that decides whether two finitely generated differential graded algebras $\M$ and $\N$ over $\Qbb$ are isomorphic. If they are isomorphic, the algorithm computes an isomorphism.
\end{mythm}

\begin{proof}
The key idea of this proof is to consider the automorphism group of the
underlying graded vector space of $\M$ and $\N$ as an algebraic group acting on the space of
all possible multiplication and differential maps and then use Sarkisyan's
algorithm to decide whether the two DGA structures lie in the same orbit of this action.

Suppose that all generators of both algebras are in degrees at most \(D\).
Choose bases of homogeneous components and identify
\[
\mathcal M^p \cong  \mathbb Q^{k_p},
\qquad
\mathcal N^p \cong  \mathbb Q^{k_p},
\]
for all \(p\le 2D\).
If the dimensions do not agree, then the algebras cannot be isomorphic.

For every \(0\le i,j\le D\), let
\[
m_{i,j}\colon
\mathcal M^i\otimes \mathcal M^j \to \mathcal M^{i+j},
\qquad
n_{i,j}\colon
\mathcal N^i\otimes \mathcal N^j \to \mathcal N^{i+j}
\]
be the multiplication maps, and for every \(0\le p<2D\), let
\[
d_{\mathcal M}^p\colon \mathcal M^p\to \mathcal M^{p+1},
\qquad
d_{\mathcal N}^p\colon \mathcal N^p\to \mathcal N^{p+1}
\]
be the differentials.

Consider the vector space
\begin{align*}
V&=
\prod_{0\le i,j\le D}
\operatorname{Hom}(\Cbb^{k_i}\otimes\Cbb^{k_j},\Cbb^{k_{i+j}})
\times\prod_{0\le p<2D}\operatorname{Hom}(\Cbb^{k_p},\Cbb^{k_{p+1}})\\
&\cong \prod_{0\le i,j\le D}\Cbb^{k_i\cdot k_j\cdot k_{o+j}}\times\prod_{0\le p<2D} \Cbb^{k_p\cdot k_{p+1}}.   
\end{align*}

Writing all maps in the chosen bases, we define

\[
x=
\left(
\prod_{0\le i,j\le D} m_{i,j},
 \prod_{0\le p<2D} d_{\mathcal M}^p
\right)
\in V
\]
and
\[
y=
\left(
\prod_{0\le i,j\le D} n_{i,j},
\prod_{0\le p<2D} d_{\mathcal N}^p
\right)
\in V.
\]

Consider the algebraic group
\[
G=\prod_{p=0}^{2D}GL(\Cbb^{k_p}).
\]

An element
\[
g=(g^0,g^1,\ldots,g^{2D})\in G
\]
acts on \(V\) by
\[
m_{i,j}\longmapsto
(g^{i+j})^{-1}\circ m_{i,j}\circ
(g^i\otimes g^j)
\]
and
\[
d_{\mathcal M}^p\longmapsto
(g^{p+1})^{-1}\circ d_{\mathcal M}^p\circ g^p.
\]

A tuple
\[
(g^0,\ldots,g^{2D})\in G
\]
determines an isomorphism of differential graded algebras
\(\mathcal M\to\mathcal N\) exactly when
\[
n_{i,j}\circ(g^i\otimes g^j)
=
g^{i+j}\circ m_{i,j}
\]
for all \(0\le i,j\le D\), and
\[
g^{p+1}\circ d_{\mathcal M}^p
=
d_{\mathcal N}^p\circ g^p
\]
for all \(0\le p<2D\).

Equivalently, \(x\) and \(y\) lie in the same \(G\)-orbit.

As \(x\) and \(y\) are finite tuples of
rational numbers, and so the above formulas define a rational representation  of
\(G\) on \(V\).

Since this action is a rational representation of an explicitly given
algebraic group, Sarkisyan's algorithm decides whether there exists
\(g\in G(\mathbb Q)\) such that
\[
x\cdot g = y.
\]

If such an element exists, the algorithm computes it explicitly, and its
components \(g^p\) form an isomorphism
\[
\mathcal M \cong \mathcal N.
\]

Therefore the isomorphism problem for finitely generated differential graded
algebras over \(\mathbb Q\) is algorithmically decidable.
\end{proof}

%% file: 05aRatHomTh.tex
\section{Preliminaries on rational homotopy theory}
Rational homotopy theory studies properties of spaces that are invariant under rational homotopy equivalence. For brevity, the notion of a space stands for a simplicial set or a topological space.
\begin{defn} [\cite{FHT}, Chapter 9(c)]
A continuous map $f\colon X\to Y$ between simply connected topological spaces is a rational homotopy equivalence
if
 one of the three equivalent conditions holds:
\begin{itemize}
    \item[(1)] $\pi_*(f)\otimes\Qbb$ is an isomorphism,
    \item[(2)] $H_*(f;\Qbb)$ is an isomorphism,
    \item[(3)]   $H^*(f;\Qbb)$ is an isomorphism.
\end{itemize}
The fact that these conditions are equivalent is the content of the Serre-Whitehead theorem. For the proof, see Theorem 8.6 in \cite{FHT}.

Let $\widetilde{X}$ and $\widetilde{Y}$ be simplicial sets. A simplicial map $f\colon \widetilde{X}\to\widetilde{Y}$ is a rational homotopy equivalence if $|f|\colon |\widetilde{X}|\to
|\widetilde{Y}|$ is a rational homotopy equivalence (equivalently, condition (2) or (3) is satisfied for $f$). 

We say that simply connected spaces $W$ and $Z$ have the same rational  homotopy type if there is a chain of rational homotopy equivalences 
\[W\leftarrow Z(0)\rightarrow\dots\leftarrow Z(k)\rightarrow Z.\]
\end{defn}

\begin{defn}
A simply connected space $Y$ is a rational space if $\pi_*(Y)$ is a $\Qbb$-module (or equivalently $H_*(Y,\Zbb)$ is a $\Qbb$-module).
A rationalization of a simply connected  space $X$ is a map $\varphi\colon X\to X_\Qbb$ to a simply connected rational space $X_\Qbb$ such that $\varphi$ induces an isomorphism $\pi_*(X)\otimes\Qbb\cong \pi_*(X_\Qbb)$.
\end{defn}

\begin{rem}
A simply connected space can be rationalized by induction using its Postnikov tower. See \cite[Section 8.2]{GM}. 
\end{rem}

The following definition provides a suitable base for an algebraic-simplicial model for simplicial sets and topolo\-gi\-cal spaces.

\begin{defn}\label{Apl}
The graded algebra $(A_{PL})_n$ is the free graded commutative algebra
\begin{align*}
    (A_{PL})_n &=\Lambda(t_0,\dots, t_n, y_0,\dots, y_n)\biggl/\left(\sum_{i} t_i-1,\sum_j y_j\right),\\
    &dt_i=y_i\quad\text{ and }\quad dy_i=0.
\end{align*}
where $t_i$ are elements of degree 0. The elements of this algebra can therefore be understood as polynomial differential forms with rational coefficients on the standard simplex $\Delta^n$.
The subspace $(A_{PL})^p_n$ of elements of degree $p$ is called the subspace of polynomial 
$p$-forms. The simplicial DGA $A_{PL}$ is the functor $\Delta^{\operatorname{op}}\to\operatorname{DGA}$ that acts on objects by the assigment $[n]\longmapsto (A_{PL})_n$ for each $[n]\in\Delta$ and any morphism $f\colon [n]\to[m]$ in $\Delta$ maps to $f^*\colon (A_{PL})_m\to (A_{PL})_n$ such that
\[f^*(t_i)=\sum_{f(j)=i} t_j \quad\text{ for } 0\leq i\leq n.\]
\end{defn}

\begin{rem}
The simplicial DGA $A_{PL}$ coincides with the graded algebra $A_{PL}=\{(A_{PL})_n\}_{n\geq0}$
together with face operators $\partial_i\colon(A_{PL})_{n+1}\to(A_{PL})_{n}$ and degeneracy operators  $s_j\colon(A_{PL})_{n}\to(A_{PL})_{n+1}$  uniquely defined by the conditions
\[\partial_i(t_k)=\begin{cases}
t_k & k<i\\
0 & k=i\\
t_{k-1} & k>i
\end{cases}\quad\quad
s_j(t_k)=
\begin{cases}
t_k & k<j\\
t_k+t_{k+1} & k=j\\
t_{k+1} & k>j
\end{cases}.\]
It is immediate to see that $\partial_i=d_i^*$ and $s_j=\rho_j^*$ where $d_i\colon [n]\to[n+1]$ is a standard face map and $\rho_j\colon [n+1]\to[n]$ is a standard degeneracy map.
\end{rem}

Let $X$ be a simplicial set. Next, we define $A_{PL}(X)$ using the following general construction applicable for any simplicial cochain algebra or simplicial cochain complex $\A$. 

\begin{constr}\label{constr}
Let $\A$ be a simplicial DGA. Then, for every simplicial set $X$, we define
\[\A(X)=\{\A^p(X)\}_{p\geq 0}\] 
as the DGA such that:
\begin{itemize}
    \item $\A^p(X)$ is the set $\Hom_{\operatorname{sSet}}(X,\A^p)$ of simplicial set homomorphisms from $X$ to $\A^p$.
    \item The addition, multiplication, scalar multiplication, and differential of $\varphi,\psi\in \Hom_{\operatorname{sSet}}(X,\A^p)$ computed on $\sigma\in X$ are given by
\[(\varphi+\psi)_\sigma =\varphi_\sigma+\psi_\sigma,\quad(\varphi\cdot\psi)_\sigma =\varphi_\sigma\cdot\psi_\sigma\quad(\lambda\cdot\psi)_\sigma =\lambda\cdot\psi_\sigma,
\quad (d\psi)_\sigma=d\psi_\sigma. \]
\end{itemize}
 A similar construction also applies to simplicial cochain complexes.  
\end{constr}

The only nondegenerate $n$-simplex $c_n=(0,1,\dots,n)$ in $\Delta[n]_n$ is called the fundamental class of $\Delta[n]$.  It turns out that the following proposition holds.

\begin{prop}\label{AplDelta}
Let $\A=\{\mathcal A_n\}$ be a simplicial DGA. For $n\geq 0$, the assignment $\phi \longmapsto \phi_{c_n}$ defines an isomorphism of DGAs $\A(\Delta[n])\to \A_n$.
In particular, $A_{PL}(\Delta[n])$ is isomorphic $ (A_{PL})_n$ for all $n\geq 0$.
\end{prop}
\begin{proof}
See Proposition 10.4  in \cite{FHT}.
\end{proof}

In Section 6 we use the fact that $\A(X)$ is a functor covariant in $\A$ and contravariant in $X$.

\begin{defn} A minimal model of a simplicial set $X$ is a minimal model for $A_{PL}(X)$.
If $T$ is a topological space then $A_{PL}(T):=A_{PL}(S_*(T))$ where $S_*(T)$ stands for the simplicial set of singular simplices.
We say that $\M$ is a minimal model for $T$ if $\M$ is a minimal model for $A_{PL}(T)$.
\end{defn}

\begin{rem}
In Section 6 we will see that there is a quasi-isomorphism $A_{PL}(X)\to C^*(X;\mathbb Q)$ of cochain complexes. 
\end{rem}


%% file: 05bRHTandDGA.tex

\subsection*{Rational homotopy theory and DGAs}

The following theorem establishes the equivalence between the homotopy category of rational spaces and the homotopy category of minimal DGAs.

\begin{prop}[Theorem 15.7. in \cite{GM}]\label{bij} 
Let $Y$ be a simply connected rational Kan complex with homotopy groups that are finite-dimensional rational vector spaces, and let $X$ be a simply connected Kan complex. Denote  $\M_X$ and $\M_Y$ associated minimal models of $X$ and $Y$, respectively. Then there is a bijection
 \[ [X,Y] \to [\M_Y,\M_X].\]
Moreover, $f\colon X\to Y$ is a rational homotopy equivalence if and only if the corresponding $\widehat{f}\colon \M_Y\to\M_X$ is an isomorphism. 
\end{prop}
\begin{proof}
Let's describe the map $[X,Y] \to [\M_Y,\M_X]$. Take a simplicial map $f\colon X\to Y$. Since $m_X$ is a quasi-isomorphism, we can apply Proposition \ref{QIsoBij} on the diagram
 \begin{center}
\begin{tikzpicture}
\matrix (m) [matrix of math nodes, row sep=2em,
column sep=4em, minimum width=2em]
{ 
 & &\M_X\\
 \M_Y & A_{PL}(Y) &  A_{PL}(X)  \\};
  \begin{scope}[every node/.style={scale=.8}]
\path[->](m-2-2) edge node[below] {$f^*$} (m-2-3);
\path[->](m-1-3) edge node[right] {$m_X$} (m-2-3);
\path[->](m-2-1) edge node[below] {$m_Y$} (m-2-2);
\path[->,dashed](m-2-1) edge node[above] {$\widehat{f}$} (m-1-3);
\end{scope}
\end{tikzpicture}
\end{center}
to get a lift $\widehat{f}\colon \M_Y\to \M_X$ unique up to homotopy. 
$f^*$ is a quasi-isomorphism if and only if $\widehat{f}$ is a  quasi-isomorphism and this is equivalent with $\widehat{f}$ being an isomorphism by Lemma \ref{QIsoIso}. 

The proof that the map $[X,Y] \to [\M_Y,\M_X]$ constructed above is a bijection is available in \cite{GM}, Chapter 15. 
\end{proof}

\begin{cor}\label{criterionKan}
Let $X$ and $Y$ be simply connected Kan complexes such that $\pi_i(X)\otimes\Qbb$ and $\pi_i(Y)\otimes\Qbb$ are finite-dimensional for all $i\in\Nbb$. Then $X$ and $Y$ have the same rational homotopy type if and only if  their minimal models are isomorphic  
\[\M_{Y}\cong\M_{X}.\]
\end{cor}

\begin{proof}
Assume that there is an isomorphism $g\colon \M_{Y}\to \M_{X}$ and consider rationalization $r\colon Y\to Y_\Qbb$ of $Y$  and its associated minimal model $\M_{Y_\Qbb}$. We can use Proposition \ref{bij} since the pairs $(Y,Y_{\Qbb})$ and  $(X,Y_{\Qbb})$ satisfy its assumptions. 
Denote $\widehat{r}\colon \M_{Y_\Qbb}\to\M_{Y}$ the image of $r$ in the bijection 
$[Y,Y_\Qbb]\cong[\M_{Y_\Qbb},\M_{Y}]$ and use the bijection $[X,Y_\Qbb]\cong[\M_{Y_\Qbb},\M_{X}]$ to get the preimage $h\colon X\to Y_\Qbb$ of the isomorphism $g\circ \widehat{r}$. Hence $h$ is a rational homotopy equivalence and the desired chain of rational homotopy equivalences is 
\[X \stackrel{h}{\longrightarrow} Y_\Qbb \stackrel{r}{\longleftarrow} Y.\]
So $X$ and $Y$ have the same rational homotopy type.

For the opposite direction, assume that there is a chain of rational homotopy equivalences 
\[X\leftarrow Z(0)\rightarrow\dots\leftarrow Z(k)\rightarrow Y.\]
$Z(i)$ can be chosen as Kan complexes. Let $f\colon W\to Z$ represent an underlying rational homotopy equivalence from the previous chain. Consider rationalization $r\colon Z\to Z_{\Qbb}$ of the Kan complex $Z$ and map the rational homotopy equivalence $r\circ f$ via bijection $[W,Z_{\Qbb}]\cong[\M_{Z_\Qbb},\M_W]$ to the isomorphism $\widehat{rf}$. Similarly, the isomorphism $\widehat{r}$ is an image of $r$ in the bijection $[Z,Z_{\Qbb}]\cong[\M_{Z_\Qbb},\M_Z]$. By composition we obtain the isomorphism $(\widehat{rf})(\widehat{r})^{-1}\colon \M_Z\to\M_W$. The required isomorphism $\M_{Y}\to\M_{X}$ is a composition of such isomorphisms or their inverses.    
\end{proof}

\begin{cor}\label{RHT}
Let $X$ and $Y$ be simply connected Kan complexes such that $\pi_i(X)\otimes\Qbb$ and $\pi_i(Y)\otimes\Qbb$ are finite dimensional rational vector spaces for all $i\in\Nbb$. Then $X$ and $Y$ have the same rational homotopy type if and only if  there is a rational homotopy equivalence $X_\Qbb\to Y_\Qbb$. 
\end{cor}

\begin{proof} If there is a rational homotopy equivalence $f:X_\Qbb\to Y_\Qbb$ then the chain of maps with rationalizations $r_X$ and $r_Y$
$$X\stackrel{r_X}\longrightarrow X_\Qbb\stackrel{f}\longrightarrow Y_\Qbb\stackrel{r_Y}\longleftarrow Y$$
determines the same homotopy type of $X$ and $Y$.

Let $X$ and $Y$ have the same rational homotopy type. According to Corollary \ref{criterionKan}
there is an isomorphism $g\colon \M_Y\to \M_X$. Due to Proposition \ref{bij} rational homotopy equivalences $r_X:X\to X_\Qbb$ and $r_Y\colon Y\to Y_\Qbb$ induce isomorphisms $\widehat{r}_X\colon\M_{X_\Qbb}\to\M_X$ and $\widehat{r}_Y\colon \M_{Y_\Qbb}\to\M_Y$. Then
\[ \M_{Y_{\Qbb}}\stackrel{\widehat{r}_Y}\longrightarrow \M_Y\stackrel{g}\longrightarrow \M_X\stackrel{{\widehat{r}_X}^{-1}}\longrightarrow \M_{X_\Qbb}  \]
is an isomorphism that induces a rational homotopy equivalence $f\colon X_\Qbb\to Y_\Qbb$ according to Proposition \ref{bij}.
  
\end{proof}

\begin{prop}\label{criterion}
Let $X$ and $Y$ be simply connected finite simplicial sets of dimensions $\le d$ with standard Postnikov towers $\{X_{n}, p_n^X,\varphi_n^X\}$ and $\{Y_{n},p_n^Y,\varphi_n^Y \}$, respectively. Then $|X|$ and $|Y|$ have the same rational homotopy type if and only if  the minimal models  of the Postnikov stages $Y_d$ and $X_d$ are isomorphic  
\[\M_{Y_d}\cong\M_{X_d}.\]
\end{prop}

\begin{proof}
 Due to Corollary \ref{criterionKan}, it suffices to prove that $|X|$ and $|Y|$ have the same rational homotopy type if and only if  their Postnikov stages $X_d$ and $Y_d$, which are Kan complexes, also have the same rational homotopy type. 

 Let us denote $X':=\varprojlim X_n$ and $Y':=\varprojlim Y_n$. They are Kan complexes and the induced maps $\varphi^X\colon X\to X'$, $\varphi^Y\colon Y\to Y'$ are weak homotopy equivalences
 and consequently rational homotopy equivalences. That is why $|X|$ and $|Y|$ have the same homotopy type if and only if $X'$ and $Y'$ also have it.

If  $X'$ and $Y'$ have the same rational homotopy type, there is a chain of rational homotopy equivalences 
\[X'\leftarrow Z(0)\rightarrow\dots\leftarrow Z(k)\rightarrow Y'\] 
with $Z(i)$ Kan complexes.  Let $f\colon W\to Z$ represent an underlying rational homotopy equivalence from the previous chain. Let $\{W_n,p_n^W,\varphi_n^W\}$ and $\{Z_n,p_n^Z,\varphi_n^Z\}$ be standard Postnikov towers of $W$ and $Z$, respectively. According to \cite[Theorem 4.2]{MS2} there is a map $f_d\colon W_d\to Z_d$ such that the diagram
\begin{center}
\begin{tikzpicture}
\matrix (m) [matrix of math nodes, row sep=2em,
column sep=2em, minimum width=2em]
{  W & Z  \\
  W_{d} & Z_{d} \\};
  \begin{scope}[every node/.style={scale=.8}]
\path[->](m-1-1) edge node[above] {$f$} (m-1-2);
\path[->](m-1-2) edge node[right] {$\varphi^{Z}_d$}  (m-2-2);
\path[->](m-1-1) edge node[left] {$\varphi^{W}_d$} (m-2-1);
\path[->](m-2-1) edge node[below] {$f_{d}$} (m-2-2);
\end{scope}
\end{tikzpicture}
\end{center}
commutes. If $f_*\colon \pi_i(W)\otimes \Qbb\to\pi_i(Z)\otimes\Qbb$ is an isomorphism for all $i$, then so is ${f_d}_*\colon \pi_i(W_d)\otimes \Qbb\to\pi_i(Z_d)\otimes\Qbb$. Hence, we obtain the chain of rational homotopy equivalences
\[X_d\leftarrow Z(0)_d\rightarrow\dots\leftarrow Z(k)_d\rightarrow Y_d'.\]
So $X_d$ and $Y_d$ have the same rational homotopy type.

Now suppose that $X_d$ and $Y_d$ have the same rational homotopy type. According to Corollary \ref{RHT} there is a rational homotopy equivalence $f_d\colon (X_d)_\Qbb\to (Y_d)_\Qbb$. Consider the diagram
 \begin{center}
\begin{tikzpicture}
\matrix (m) [matrix of math nodes, row sep=2em,
column sep=4em, minimum width=2em]
{ 
 & &Y_\Qbb\\
 X & (X_d)_\Qbb &  (Y_d)_\Qbb  \\};
  \begin{scope}[every node/.style={scale=.8}]
\path[->](m-2-2) edge node[below] {$f_d$} (m-2-3);
\path[->](m-1-3) edge node[right] {$\varphi_d^Y$} (m-2-3);
\path[->](m-2-1) edge node[below] {$r_{X_d}\circ\varphi_d^X$} (m-2-2);
\path[->,dashed](m-2-1) edge node[above] {} (m-1-3);
\end{scope}
\end{tikzpicture}
\end{center}
where $r_{X_d}\colon X_d\to (X_d)_\Qbb$ is a rationalization of $X_d$.
Since $|X|$ is a CW-complex of dimension $d$ and $|\varphi_d^Y|:|Y_\Qbb|\to |(Y_d)_\Qbb|$ is an isomorphism in homotopy groups $\pi_i$ for $i\le d$ and an epimorphism for $i=d+1$, the induced map 
\[|\varphi_d^Y|_*\colon [|X|,|Y_\Qbb|]\to [|X|,|(Y_d)_\Qbb|]\] 
is a bijection.
Let $F\colon |X|\to |Y_\Qbb|$ be a map which homotopy class corresponds to the homotopy class of $|f_d\circ r_{X_d}\circ\varphi_d^X|\colon
|X|\to |(Y_d)_\Qbb|$. Then $F$ induces isomorphisms in rational homology groups $H_i(\ ;\Qbb)$ for $i\le d$. Since $|X|$ and $|Y|$ have dimensions $\le d$, the homology groups $H_i(X;\Qbb)$ and $H_i(Y_\Qbb;\Qbb)$ are trivial for $i>d$ and hence $F$  induces isomorphism on all homology groups
for all $i$. Therefore we have a chain of rational homotopy equivalences
\[F\colon |X|\stackrel{F}\longrightarrow |Y_\Qbb|\stackrel{|r_Y|}\longleftarrow |Y|\] 
and so $|X|$ and $|Y|$ have the same rational homotopy type.
 \end{proof}

This result leads us to investigate the computability of minimal models for Postnikov stages.

%% file: 05cMinModAndPostnikTow.tex

\subsection*{Minimal models for Postnikov stages}
The basic property of minimal algebras is that they are increasing sequences of subalgebras that are related to each other via Hirsch extensions.

Throughout the chapter, we denote the associated chain complex to a rational vector space $V$ located in degree  $n$ as $V[n]$. Denote the dual of $V$  as $V^*$.

\begin{defn}
Let $(\A,\partial)$ be a DGA. A Hirsch extension of $\A$ is a DGA
\[ (\A\otimes_d \Lambda(V[n]), D)\]
where
\begin{itemize}
    \item[(i)] $V$ is a (finite-dimensional) vector space,
    \item[(ii)] $\Lambda(V[n])$ is the free graded-commutative algebra generated by $V$ in degree $n$,
    \item[(iii)] $\A\otimes\Lambda(V[n])$ is an underlying graded algebra with identifications $\A=\A\otimes 1$ and $V=1\otimes V$,
    \item[(iv)] $d\colon V\to \A^{n+1}$ is a homomorphism of vector spaces with $\partial(\im d)=0$,
    \item[(v)] the differential $D$ on $\A\otimes_d \Lambda(V[n])$ is determined by its restrictions \[D|_\A=\partial \ \text{ and }\ D|_{V}=d.\]
\end{itemize}
\end{defn}

The following theorem explains why there is a certain type of duality between principal minimal fibrations and Hirsch extensions.

\begin{thm}[\cite{GM}, Theorem 12.1]\label{mod1}
Let $B$ and $E$ be simplicial sets and let $f\colon E\to B$ be a principal $K(\pi,n)$ Kan fibration with
a Postnikov class $[\kappa]\in H^{n+1}(B;\pi)$ where $\pi$ is an abelian group and $V=\pi\otimes\Qbb$ is a finite-dimensional rational vector space.
Let $\rho_B\colon \M_B\to A_{PL}(B)$ be a  minimal model for $B$. Denote $\M_E=\M_B\otimes_{d}\Lambda(V^*[n])$ the Hirsch extension with a differential $d$ such 
that 
\[[d]\in \Hom(V^*,H^{n+1}(\M_B))\] 
is identified through $\rho_B$  with 
\[[\kappa\otimes 1]\in H^{n+1}(B;\pi)\otimes\Qbb\cong H^{n+1}(B;V).\]
Then, there is a map $\rho_E\colon \M_E\to A_{PL}(E)$ which is a  minimal model for $E$ and makes
commutative the following diagram:
\begin{center}
\begin{tikzpicture}
\matrix (m) [matrix of math nodes, row sep=2em,
column sep=2em, minimum width=2em]
{  \M_B & A_{PL}(B)  \\
  \M_E& A_{PL}(E)\\};
  \begin{scope}[every node/.style={scale=.8}]
\path[<-<](m-2-1) edge   (m-1-1);
\path[<-](m-1-2) edge node[below] {$\rho_B$} (m-1-1);
\path[<-](m-2-2) edge node[right] {$f^*$} (m-1-2);
\path[<-](m-2-2) edge node[below] {$\rho_E$}(m-2-1);
\end{scope}
\end{tikzpicture}
\end{center}
\end{thm}
In this case, we say that Kan fibration $E\to B$ and the Hirsch
extension $\M_E$ of  $\M_B$ are dual.

\begin{proof}
Here, we only explain the relation between $[\kappa]$ and $[d]$ if $H^{n+1}(B;\Qbb)$ 
is a rational vector space of finite dimension. 
\begin{align*}
[d]\in &\operatorname{Hom}(V^*,H^{n+1}(\M_B))
\cong \operatorname{Hom}(V^*,H^{n+1}(A_{PL}(B)))
\cong\operatorname{Hom}(V^*,H^{n+1}(B;\Qbb))\\
&\cong\operatorname{Hom}(H_{n+1}(B;\Qbb),V)
\cong H^{n+1}(B;V)\ni [\kappa\otimes 1]
\end{align*}
\end{proof}

A direct consequence of the previous theorem is that by repeating this procedure, we obtain a minimal model for the $n$-th Postnikov stage $X_n$.
As an induction base, we can use the minimal model $(\Lambda(V^*[2]),0)$ for $K(\pi_2(X),2)$ with
\[V=\pi_2(X)\otimes_\Zbb\Qbb.\]

\begin{cor}[\cite{GM}, Corollary 12.2]\label{ModCor}
Suppose that $X$ is a simply connected simplicial set whose rational homology is a finite-dimensional rational vector space in each dimension with a standard Postnikov tower $\{X_n\}$.
Let $\M$ be a  minimal model for $X$ and let $\M(n)$ be a minimal sub-DGA generated by elements of $\M$ 
in degrees $\leq n$. Then, $\M(n)$ is a finitely generated minimal model for $X_n$. In particular, $\M(n)$ is the Hirsch extension of $\M(n-1)$ dual to the principal $K(\pi_n(X),n)$ Kan fibration $X_n\to X_{n-1}$.
\end{cor}

%% file: 06MinModForSSet.tex

\section{Algorithmic minimal model  for a simplicial set}
In this section, we show that if $X$ is a simply connected simplicial set with effective homology $C_*^{\ef}(X,\Qbb)$, then $A_{PL}(X)$ satisfies the assumptions of Theorem \ref{minMod}. That means that one can compute an $n$-minimal Sullivan model for such a simplicial set $X$ for every $n\in \Bbb N$.

The aim lies in transporting the effective homology of $C^*_{\ef}(X,\Qbb)$ to $A_{PL}(X)$ through the chain of isomorphisms and reductions of rational cochain complexes.
\begin{center}
\begin{tikzpicture}
\matrix (m) [matrix of math nodes, row sep=2em,
column sep=3em, minimum width=2em]
{ 
C^*_{\ef}(X,\Qbb)  & C^*(X,\Qbb)  & W(X) & A_{PL}(X) \\
};
  \begin{scope}[every node/.style={scale=.8}]
\path[->](m-1-1) edge node[above] {$\Longleftrightarrow$}    (m-1-2);
\path[->](m-1-1) edge node[below] {\circled{0}}    (m-1-2);
\path[->](m-1-2) edge node[above] {$\cong$}   (m-1-3);
\path[->](m-1-2) edge node[below] {\circled{1}}   (m-1-3);
\path[->](m-1-3) edge node[above] {$\Longleftarrow$} (m-1-4);
\path[->](m-1-3) edge node[below] {\circled{2}} (m-1-4);
\end{scope}
\end{tikzpicture}
\end{center}
where $C^*(X,\Qbb)$ is the cochain complex of $X$ and $W(X)$ are Whitney forms on $X$.
The composition of maps $\circled{1}$ and $\circled{2}$ already appeared in Dupont's work \cite{JD2}
and it is known as the explicit simplicial de Rham theorem.


\subsection*{Effective homology for $C^*(X,\Qbb)$}

Assume that we have a strong homotopy equivalence of chain complexes  
$C_*(X)\Longleftrightarrow C_*^{\ef}(X)$. Then 
\[C^*(X,\Qbb) =(C_*(X)\otimes\Qbb)^*\Longleftrightarrow (C_*^{\ef}(X)\otimes\Qbb)^*=C^*_{\ef}(X,\Qbb)\] 
by Lemma \ref{HomEf}.


\subsection*{The isomorphism between $C^*(X,\Qbb)$ and $W(X)$} We will define cochain complexes $C_{PL}(X)$ and $W(X)$ and show that the isomorphism \circled{1} is a composition of two natural isomorphisms
\[C^*(X;\Qbb)\cong C_{PL}(X)\cong W(X).\]

Recall Construction \ref{constr} and Proposition \ref{AplDelta}. Their consequence is that $\mathcal (A_{PL})_n$ are polynomial differential forms on the standard $n$-simplex.
Similarly, $\Delta[n]$ can be substituted in $C^*(-)$ to obtain a simplicial cochain algebra $C_{PL}$ (more details are available in the book \cite{FHT}).

\begin{defn}\label{Cpl} 
The simplicial cochain algebra $C_{PL}=\{(C_{PL})_n\}_{n\geq 0}$ is defined by
\begin{itemize}
\item cochain  algebras $(C_{PL})_n=C^*(\Delta[n];\Qbb)$ for $n\ge 0$,
\item the face $\partial_i=C^*([d_i])$ and degeneracy $s_j=C^*([\rho_j])$ morphisms which are induced by standard inclusions $d_i\colon [n]\to [n+1]$ and standard degeneracy maps $\rho_j\colon [n+1]\to [n]$.
\end{itemize}
Now, Construction \ref{constr} defines the cochain algebra $C_{PL}(X)$.
\end{defn}

\begin{thm}\label{CCpl}
Let $X$ be a simplicial set. Then there is a natural isomorphism $C_{PL}(X)\to C^*(X;\Qbb)$
of cochain algebras.
\end{thm}
\begin{proof}
The proof is due to Watkins (\cite{CW}) and is available in \cite[Lemma 10.11]{FHT}.
From an algorithmic point of view, it is sufficient to prescribe the map; all other details are available in the referenced literature.
For $p\geq 0$, take $\gamma\in C_{PL}^p(X)$ and $\sigma\in X_p$. Then $\gamma_{\sigma}\in C^p(\Delta[p];\Qbb)$
and we can define $g\in C^p(X;\Qbb)$ as
\[ g(\sigma)=\gamma_{\sigma}(c_p).\]
The inverse map $C^p(X;\Qbb)\to C_{PL}(X)$ is defined as follows. Each $\sigma\in X_n$ determines a unique simplicial map $\sigma_*\colon \Delta[n]\to X$ such that $\sigma_*(c_n)=\sigma$. So, the inverse map is
\[g\longmapsto \gamma \text{ such that } \gamma_{\sigma}=C^p(\sigma_*)(g) \text{ for } \sigma\in X_n.\]
\end{proof}

\subsubsection*{Whitney forms}
The Whitney forms originate in the book of Hassler Whitney \cite{HW}, he called them elementary forms. Generally, Whitney forms are differential forms in a simplicial complex $K$ embedded in an affine space $\Rbb^{n+1}$, see \cite{LOHI}. For our purposes, we focus only on Whitney forms in the standard $n$-simplex $\Delta^n$ (see \cite{LL}) and the extension of the definition for every simplicial set via general construction.

\begin{defn}
The Whitney form associated to $f\colon [p] \to [n]$ in $\Delta$ is the $p$-form
\[\omega_f = p!\sum_{i=0}^p (-1)^i t_{f(i)}\d{t_{f(0)}}\dots\widehat{\d{t_{f(i)}}}\dots \d{t_{f(p)}}\in (A_{PL})_n^p.\]
The Whitney $p$-forms $W^p_n$  is a vector subspace in $(A_{PL})_n^p$ spanned by all Whitney forms associated with the morphisms $f\colon [p] \to [n]$ in $\Delta$. The Whitney forms determine the subspace $W_n=\{(W)_n^p\}_{p\geq0}$ of $(A_{PL})_n$.
\end{defn}
If $f$ is not an injective morphism, then the Whitney form associated with $f$ is trivial, that is, $\omega_f=0$, see \cite[Proposition 3.4 (1)]{LL}. 
The vector spaces $W_n$ are finite-dimensional. 

\begin{prop}
The space $W:=\{W_n\}_{n\geq 0}$ of the Whitney forms is a simplicial DG-vector subspace of $A_{PL}$. 
\end{prop}
\begin{proof}
See Proposition 3.4 in \cite{LL}.
\end{proof}
This proposition enables us to apply Construction \ref{constr} on $W$, so we obtain a notion of Whitney forms $W(X)$ for any simplicial set $X$. Proposition \ref{AplDelta} says that $W_n\cong W(\Delta[n])$.\\

To remove trivial Whitney forms, we introduce the set $I(p,n)\subseteq\operatorname{Mor}_{\Delta}([p],[n]])$ of all injective and strictly monotone maps in the category of finite ordinals $\Delta$.\\

\emph{Integration over simplices.} 
Integrals are defined using axiomatic rules derived from the classical formula for integration by parts for the Riemann integral. More details are available in \cite[p. 128]{FHT}.

\begin{defn}
Let $\sigma=(i_0,i_1,\dots,i_k)$ be a nondegenerate $k$-simplex in $\Delta[n]$. Denote $\Delta^\sigma$ the geometric $k$-simplex with vertices $e_{i_0}, e_{i_1},\dots,e_{i_k}$.
The integration map on $\Delta^\sigma$ is the map
\[\int_{\Delta^\sigma}\colon (A_{PL})_n\to\Qbb\]
defined by linearity using the two identities:
\[ \int_{\Delta^\sigma} t_{i_0}^{a_0}t_{i_1}^{a_1}\dots t_{i_k}^{a_k} \d{t_{i_0}}\wedge\dots\wedge\widehat{\d{t_{i_j}}}\wedge\dots\wedge\d{t_{i_k}}
=(-1)^j\frac{a_{0}!\dots a_{k}!}{(a_{0}+\dots+a_{k}+k)!},\]
and
\[ \int_{\Delta^\sigma}\eta = 0 \text{ if } \eta\in (A_{PL})^p_n, \text{ with } p\ne k, \text{ or } \eta=0 \text{ on } \Delta^\sigma.\]

\end{defn}

The first formula explains the role of $p!$ in the definition of Whitney forms $\omega_f$ for $f\colon [p]\to [n]$, because it forces $\int_{\Delta^\sigma} f^*\omega_f=1$ for $\sigma=(f(0),f(1),\dots, f(p))$.

In the next step, our aim is to identify $(C_{PL})_n$ with $W_n$ using the integration map. 
Since $k$-forms are integrable over $k$-simplices, each $k$-form $\omega$ produces a $k$-cochain whose values on the chains are integrals of $\omega$. 
\begin{defn}
The de Rham map $DR\colon W_n^k\to (C_{PL})_n^k$ is a linear map defined by setting:
\[ DR(\omega)\left(\sum_i a_i\sigma_i\right) = \sum_i a_i\int_{\Delta^{\sigma_i}}\omega.\]

\end{defn}
Note that we use the same notation $d$ for coboundary  $C^p(\Delta[n];\Qbb)\to C^{p+1}(\Delta[n];\Qbb)$ and  the exterior derivative of the differential forms. Then, Stokes' theorem implies that 
\[DR(d)=d(DR).\] 

For the opposite map $(C_{PL})_n^k\to W_n^k$, we need to associate with every generator of $C^k(\Delta[n];\Qbb)$ a map $f\colon [k]\to[n]$. As $\Delta[n]$ is a finite simplicial complex, $C^k(\Delta[n];\Qbb)$ is a finite-dimensional vector space, and so $C^k(\Delta[n];\Qbb)\cong C_k(\Delta[n];\Qbb)$, i.e., every nondegenerate $k$ -simplex $\sigma=(i_0,i_1,\dots,i_k)$ uniquely determines the cochain $\widetilde\sigma_{i_0,i_1,\dots,i_k}$ of the dual basis in $C^k(\Delta[n];\Qbb)$ such that $\widetilde\sigma(\sigma)=1$ and $\widetilde\sigma(\tau)=0$
for all $\tau\ne\sigma$.

\begin{defn}
The Whitney map $WH\colon (C_{PL})_n^k\to W_n^k$  is a linear map defined uniquely by setting
\[WH(\widetilde\sigma_{i_0,i_1,\dots,i_k})=\omega_f \text{ such that } f(j)=i_j \text{ for each } j=0,1,\dots,k.\]
\end{defn}

\begin{prop}\label{iso} 
The map $WH:(C_{PL})_n\to W_n$ is an isomorphism such that
\[DR\circ WH=\operatorname{id}_{{(C_{PL})}_n}.\]
\end{prop}
\begin{proof}
See the proof of Proposition 4.1 in \cite{LOHI}.
\end{proof}

Now, we need to show that the maps $WH$ and $DR$ respect simplicial structures and so they can be considered as isomorphisms between simplicial cochain complexes $W$ and $C_{PL}$.

\begin{thm}
The maps $WH\colon C_{PL}\to W$ and $DR\colon W\to C_{PL}$ are inverse simplicial isomorphisms to each other. 
\end{thm}

\begin{proof}
$DR$ and $WH$ are linear isomorphisms according to Proposition \ref{iso}. Since $DR$ is an isomorphism of cochain complexes, its inverse has to be also such an isomorphism.  The fact that $DR$ commutes with face $\partial_i$ and degeneracy operators $s_j$ has been proved in \cite[Theorem 10.15(i)]{FHT}. This again implies that its inverse $WH$ also commutes with these operators.
\end{proof}

\begin{cor}
Let $X$ be a simplicial set. Then
\[C^*(X;\Qbb)\cong C_{PL}^*(X)\cong W^*(X).\]
\end{cor}

\begin{proof}
The previous theorem says that $W\cong C_{PL}$ is a simplicial isomorphism, so it also induces $C_{PL}^*(X)\cong W^*(X)$ as DG vector spaces for any simplicial set $X$.
Moreover, Theorem \ref{CCpl} provides $C^*(X;\Qbb)\cong C_{PL}^*(X)$.    
\end{proof}

\subsection*{The reduction of $A_{PL}(X)$ to $W(X)$}
This section describes the maps providing the reduction of $A_{PL}$ to $W$.

\begin{thm}[Dupont, \cite{JD}, Getzler, \cite{EG}] For each $m\ge 0$ consider the inclusion $i_m\colon W_m\to (A_{PL})_m$
and the operator $\pi_m\colon (A_{PL})_m^*\to W_m^*$ 
\[\pi_m(\eta)=\sum_{p=0}^m\sum_{f\in I(p,m)}\left(\int_{\Delta^p}f^*\eta\right)\omega_f.\]
There exist Dupont homotopies 
\[h_{\Delta[m]}\colon (A_{PL})_m^*\to (A_{PL})_m^{*-1}\]  
such that
\begin{itemize}
    \item the operator $\pi_m$ is a simplicial projector onto $W_m$ i. e. $\pi_mi_m=\id$,
    \item  $h_{\Delta[m]}d+dh_{\Delta[m]}=i_m\pi_m-\id$ and $h_{\Delta[m]}$ is a simplicial map,
    \item $h_{\Delta[m]}^2=0$, $\pi_mh_{\Delta[m]}=0$ and $h_{\Delta[m]}i_m=0$.
\end{itemize}
The triple $(\pi_m,i_m,h_{\Delta[m]})$ is a simplicial reduction of DG vector spaces. 
Moreover, the Dupont homotopies are constructable in an algorithmic way.
\end{thm}

\begin{proof}
 See the proof of Theorem 2.3 in \cite{JD}  or the proof Lemma 3.4,  Theorem 3.7 and Theorem 3.11 in \cite{EG}.   
\end{proof}

\subsubsection*{Dupont homotopy construction}
The principle of Dupont homotopy comes from the idea of the proof of Poincar\'{e} lemma.
First, we define the map $f_j\colon [0]\to [m]$ with the assignment $f_j(0)=j$ and associate it with the dilation map $\widehat{f}_j\colon \Delta^1\times\Delta^m\to\Delta^m$ 
\[\widehat{f}_j(s,v)=se_j+(1-s)v\]
where $s\in [0,1]$, $v\in\Delta^m$.
For any $\eta\in A_{PL}(\Delta[m])$, there are unique forms $\alpha_{\eta}^j,\beta_{\eta}^j\in A_{PL}(\Delta[m])[s]$ (that is, forms on $\Delta^m$ the coefficients of which are polynomials in the coordinates of $\Delta^m$ and $s$) such that
\[\widehat{f}_j^*(\eta)=(\d{s})\alpha_\eta^j+\beta_\eta^j\] 
where $\widehat{f}_j^*\colon A_{PL}(\Delta[m])\to A_{PL}(\Delta[1])\otimes A_{PL}(\Delta[m])$ is the pullback of $\widehat{f}_j$. Dupont defined $h_j\in\Hom_\Qbb^{-1}(A_{PL}(\Delta[m]),A_{PL}(\Delta[m]))$ by 
\[h_j(\eta) = \int_0^1\alpha_\eta^j\d{s}.\]
If
\[\alpha_\eta^j = (1-s)^as^bt_0^{k_0}t_1^{k_1}\dots t_m^{k_m}\d{t_{c_1}}\dots\d{t_{c_l}},\]
then
\[h_j(\eta) = \left(\int_0^1(1-s)^as^b\d{s}\right)t_0^{k_0}t_1^{k_1}\dots t_m^{k_m}\d{t_{c_1}}\dots\d{t_{c_l}}.\]
Consider this map as defined for the map $f_j$. We can extend the definition for all maps $f\in I(p,m)$ by setting
\[h_f=h_{f(p)}\circ \dots\circ h_{f(0)}\in \Hom_\Qbb^{-p-1}(A_{PL}(\Delta[m]),A_{PL}([\Delta[m])).\] 
Now, the Dupont homotopy $h_{\Delta[m]}\in\Hom_\Qbb^{-1}(A_{PL}(\Delta[m]),A_{PL}(\Delta[m]))$ is defined by
\[h_{\Delta[m]}(\eta)=\sum_{p=0}^{m-1}\sum_{f\in I(p,m)} \omega_f h_f(\eta).\]

\begin{rem}
A careful reader can observe different summation upper bounds in the definition of Dupont homotopy and Dupont projection. The reason is that $I(m,m)=\{\id\}$ and $h_{\id}$ reduces the degree of the form by $(m+1)$, therefore $h_{\id}=0$.

\end{rem}

\begin{exmp}
Consider $f\colon [1]\to [2]$  defined as $f(0)=1, f(1)=2$  and $\eta =t_1^2\d{t_2}$. Our aim is to compute $h_f(\eta)$. At first, we compute $\widehat{f}_1^*(\eta)$ so we need to make the substitution $t_1=(1-s)x_1+s$ and $t_2=(1-s)x_2$ according to $\widehat{f}_1$. Thus, we have 
\[\widehat{f}_1^*(\eta)=[(1-s)^2x_1^2+2(1-s)sx_1+s^2][-x_2\d{s}+(1-s)\d{x_2}]\]
and 
\begin{align*}
\widetilde{\eta}=h_1(\eta)&=-\int_0^1 (1-s)^2\d{s}x_1^2x_2-2\int_0^1 (1-s)s\d{s}x_1x_2-\int_0^1 s^2\d{s}x_2\\
 &=-\frac13x_1^2x_2-\frac13x_1x_2-\frac13x_2.
\end{align*}
As there is no term $\d{x_0},\d{x_1},\d{x_2}$ in the form $\widetilde{\eta}$, we will get $\alpha_{\widetilde{\eta}}=0$ and so $h_2(\widetilde{\eta})=0$. Finally, this implies that $h_f(\eta)=h_2(h_1(\eta))=0$.
\smallskip

Let us compute 
\[h_{\Delta[2]}(\eta)=\sum_{p=0}^{1}\sum_{f\in I(p,2)} \omega_f h_f(\eta). \] 
Since for $f\in I(1,2)$ we get $h_f(\eta)=0$, it suffices to consider only $f_i\in I(0,2)$:

\begin{center}
\begin{tabular}{c|c|c|c}
$f_i$ & $\omega_{f_i}$ & $h_{f_i}$ & $\omega_{f_i}h_{f_i}(\eta)$ \\\hline\rule{0pt}{2.5ex}    
$f_0(0)=0 $ & $t_0=1-t_1-t_2$ & $-\frac{1}{3}t_1^2t_2$ &$ -\frac{1}{3}t_1^2t_2 + \frac{1}{3}t_1^2t_2^2 + \frac{1}{3}t_1^3t_2$ \\\rule{0pt}{2.5ex} 
$f_1(0)=1, $   & $t_1$ & $ -\frac{1}{3}t_2 -\frac{1}{3}t_1t_2 -\frac{1}{3}t_1^2t_2$ &$ -\frac{1}{3}t_1t_2 -\frac{1}{3}t_1^2t_2 -\frac{1}{3}t_1^3t_2$ \\
$f_2(0)=2, $ & $t_2$ &
$ \frac{1}{3}t_1^2 - \frac{1}{3}t_1^2t_2$ &$ \frac{1}{3}t_1^2t_2 -\frac{1}{3}t_1^2t_2^2$ \\
\end{tabular}
\end{center}
The results is
\[h_{\Delta[2]}(\eta)=-\frac{1}{3}t_1t_2 -\frac{1}{3}t_1^2t_2.\]
To show that $\d h_{\Delta[2]}(\eta)+h_{\Delta[2]}\d(\eta)=\pi_2(\eta)-\eta$ we have compute
 $h_{\Delta[2]}(\d\eta)$. Here, calculations are much more complicated since we also use maps $f^a, f^b, f^c\in I(1,2)$, $f^a(0)=0,f^a(1)=1$, $f^b(0)=0,f^b(1)=2$, 
 $f^c(0)=1,f^c(1)=2$. 
 \begin{align*}
\omega_{f_0}h_{f_0}(d\eta)&=\frac{2}{3}t_1t_2\d{t_1} -\frac{2}{3}t_1t_2^2\d{t_1} -\frac{2}{3}t_1^2\d{t_2} + \frac{2}{3}t_1^2t_2\d{t_2} -\frac{2}{3}t_1^2t_2\d{t_1}+ \frac{2}{3}t_1^3\d{t_2},\\
\omega_{f_1}h_{f_1}(d\eta)&=\frac{1}{3}t_1\d{t_2} + \frac{1}{3}t_1t_2\d{t_1} + \frac{1}{3}t_1^2\d{t_2} + \frac{2}{3}t_1^2t_2\d{t_1} -\frac{2}{3}t_1^3\d{t_2},\\
\omega_{f_2}h_{f_2}(d\eta)&=-\frac{2}{3}t_1t_2\d{t_1} + \frac{2}{3}t_1t_2^2\d{t_1} -\frac{2}{3}t_1^2t_2\d{t_2},\\
\omega_{f^a}h_{f^a}(d\eta)&= \frac{1}{3}t_2\d{t_1} -\frac{1}{3}t_2^2\d{t_1} + \frac{1}{3}t_1t_2\d{t_2} + \frac{1}{3}t_1t_2\d{t_1} -\frac{1}{3}t_1t_2^2\d{t_1} + \frac{1}{3}t_1^2t_2\d{t_2},\\
\omega_{f^b}h_{f^b}(d\eta)&=-\frac{1}{3}t_1^2\d{t_2} -\frac{1}{3}t_1^2t_2\d{t_1} + \frac{1}{3}t_1^3\d{t_2},\\
\omega_{f^c}h_{f^c}(d\eta)&=-\frac{1}{3}t_2\d{t_1} + \frac{1}{3}t_2^2\d{t_1} + \frac{1}{3}t_1\d{t_2} -\frac{1}{3}t_1t_2\d{t_2} + \frac{1}{3}t_1t_2^2\d{t_1} -\frac{1}{3}t_1^2t_2\d{t_2} \\
&+\frac{1}{3}t_1^2t_2\d{t_1} -\frac{1}{3}t_1^3\d{t_2}.
\end{align*} 
Summary results are:
\begin{align*}
h_{\Delta[2]}(\d\eta)&= \frac{2}{3}t_1\d{t_2} + \frac{2}{3}t_1t_2\d{t_1} -\frac{2}{3}t_1^2\d{t_2},\\
\d h_{\Delta[2]}(\eta)&=-\frac{1}{3}t_2\d{t_1} -\frac{1}{3}t_1\d{t_2} -\frac{2}{3}t_1t_2\d{t_1} -\frac{1}{3}t_1^2\d{t_2},\\
\pi_2(\eta)&=\left(\int_{\Delta[1]} (f^c)^*(\eta)\right)\omega_{f^c}= \int_{\Delta[1]} t_0^2\d{t_1}\omega_{f^c}=  \int_{\Delta[1]} (1-t_1)^2\d{t_1}\omega_{f^c} =\frac{1}{3}(t_1\d{t_2} - t_2\d{t_1}),\\
\eta &= t_1^2\d{t_2}.
\end{align*}
It's clear that $dh_{\Delta[2]}(\eta) + h_{\Delta[2]}(d\eta) = \pi_2(\eta)-\eta$.
\end{exmp}

\begin{cor}
Let $X$ be a simplicial set. Then, there is a reduction 
\[A_{PL}(X)\Longrightarrow W(X).\]
\end{cor}
\begin{proof}
The statement is a direct consequence of the previous theorem. The homomorphisms describing the reduction $A_{PL}\Longrightarrow W$ are simplicial, therefore, the induced maps between $A_{PL}(X)=\Hom_{\operatorname{sSet}}(X,A_{PL})$ and $W(X)=\Hom_{\operatorname{sSet}}(X,W)$ also form a reduction.
\end{proof}

\subsection*{The simplicial de Rham Theorem over $\Qbb$}
This part focuses on an explicit description of reduction $A_{PL}(X)\Longrightarrow C^*(X;\Qbb)$ derived by Dupont \cite[Chapter 2]{JD2}. Dupont mentions that Dennis Sullivan obtained a similar result. 

\begin{thm}[Dupont, \cite{JD2}]\label{deRham}
There are natural chain maps $\Ee\colon C^*(X;\Qbb)\to A^*_{PL}(X)$ and  $\I\colon A^*_{PL}(X)\to C^*(X;\Qbb)$
together with natural chain homotopies
$\Ss\colon A^*_{PL}(X)\to A^{*-1}_{PL}(X)$, 
such that 
\begin{align*}
\I\circ \Ee-\id &= 0,\\
\Ss\circ d+d\circ \Ss&=\Ee\circ \I - \id.
\end{align*}
Moreover, homotopy $\Ss$ has the properties $\Ss^2=0$ and $\I\circ \Ss=0$ (and so $\Ss\circ E=0$). 
 That means that \[A_{PL}(X)\Longrightarrow C^*(X;\Qbb).\]
\end{thm}
\begin{proof}
Here, we summarize key constructions from the proof of Theorem 2.16 in \cite{JD2}. 
For algorithms, one needs to supply the definitions o maps. (The remaining formalities are available in the referenced proof.)
\begin{align*}
 \I(\varphi)_\sigma&=\int_{\Delta^k}\varphi_\sigma,\quad \varphi\in A^k_{PL}(X), \sigma\in X_k,\\
 \Ee(\psi)_\sigma&=\sum_{f\in I(k,m)} \omega_f\cdot \psi_{f^*(\sigma)}, \quad \psi\in C^k(X;\Qbb),\sigma\in X_m\\
 \Ss(\varphi)_\sigma&= \sum_{p=0}^{k-1}\sum_{f\in I(p,k)}\omega_f h_f(\varphi_\sigma), \quad \varphi\in A^k_{PL}(X),\sigma\in X_m,
\end{align*}
where $f^*\colon X_k\to X_m$ is a simplicial map induced by $f$.
The additional properties of $\Ss$ are due to Lemma 3.4 part i) and Theorem 3.11 in \cite{EG}.
\end{proof}

One can observe that the composition $\Ee\circ \I$ is in a relationship with the definition of $\pi_m$ and $\Ss$ is related to $h_{\Delta[m]}$ from the previous subsection. It is straightforward to verify that the reduction $(\I,\Ee,\Ss_k)$ is a composition of the isomorphism $\circled{1}$ and the reduction $\circled{2}$.

%% file: 07mainThm.tex

\section{Main results}

We are now ready to formulate and briefly prove our main results.

\begin{mythm}{C}\label{minModA}
 There is an algorithm that, for a simply connected simplicial set $X$ with effective homology and for a number $d\in\Nbb$, computes the minimal model of $X$ up to degree $d$.
\end{mythm} 

\begin{proof}
The effective homology of $X$ provides the strong homotopy equivalence $C^*(X;\Qbb)$ $\Longleftrightarrow C^*_{\ef}$ where $C^*_{\ef}=(C_*^{\ef}\otimes \mathbb Q)^*$ and $C_*^{\ef}$ is  effective homology of $X$. Theorem \ref{deRham} provides a reduction $A_{PL}(X)\Longrightarrow C^*(X;\Qbb)$. Then, we compose both strong homotopy equivalences into one strong homotopy equivalence $C^*_{ef}\Longleftrightarrow A_{PL}(X)$.
Note that $A_{PL}(X)$ is simply connected as the same property has the simplicial set $X$. Now, we can apply Theorem \ref{minMod} to get a minimal model of $X$ up to degree $d$.
\end{proof}

\begin{mythm}{D}\label{main}
 There is an algorithm that decides, for given finite simply connected simplicial sets $X$ and $Y$, whether $|X|$ and $|Y|$ have the same rational homotopy type.    
\end{mythm}

\begin{proof}
 Let $d$ be an integer such that $\dim X\le d$ and $\dim Y\le d$.  Using the algorithm of Theorem \ref{minModA} we construct their minimal models $\M_X(d)$ and $\M_Y(d)$ up to degree $d$.
 According to Corollary \ref{ModCor} they are minimal models of Postnikov stages $X_d$ and $Y_d$ of $X$ and $Y$, respectively. Once the minimal models of $X_d$ and $Y_d$ are available, we determine if they are isomorphic by applying the algorithm in Theorem \ref{IsoMod}. Then we apply Proposition \ref{criterion} to decide whether $X$ and $Y$ have the same rational homotopy type.
\end{proof}